\documentclass[xcolor=dvipsnames,svgnames,table,reqno]{amsart}

\input xy
\xyoption{all}
\usepackage{accents}
\usepackage{epsfig}
\usepackage{xcolor}
\usepackage{amsthm}
\usepackage{amssymb}
\usepackage{bbm}
\usepackage{amsmath}
\usepackage{amscd}
\usepackage{amsopn}
\usepackage{graphicx}
\usepackage{xspace}

\usepackage{hhline}
\usepackage{easybmat}
\usepackage{caption}   
\usepackage{relsize}

\usepackage{url}
\usepackage{enumitem, hyperref}\hypersetup{colorlinks}

\usepackage{stmaryrd}

\usepackage[T1]{fontenc}

\colorlet{purpleB70}{blue!70!red}

\colorlet{orangeR65}{red!65!yellow}

\definecolor{red2}{HTML}{d41173}

\definecolor{neongreen}{HTML}{1bf702}

\definecolor{radicalred}{HTML}{FF355E}

\definecolor{denim}{HTML}{1560BD}

\definecolor{darkcyan}{rgb}{0.0, 0.55, 0.55}

\definecolor{cilek}{HTML}{FF43A4}

\definecolor{mor}{HTML}{9F00C5}

\definecolor{phlox}{rgb}{0.87, 0.0, 1.0}

\definecolor{fluorescentpink}{HTML}{FF1493}

\definecolor{napiergreen}{rgb}{0.16, 0.5, 0.0}

\definecolor{kellygreen}{rgb}{0.3, 0.73, 0.09}

\definecolor{parisgreen}{HTML}{ 50C878 }

\definecolor{palatinateblue}{rgb}{0.15, 0.23, 0.89}

\definecolor{ceruleanblue}{rgb}{0.16, 0.32, 0.75}

\definecolor{brandeisblue}{rgb}{0.0, 0.44, 1.0}

\definecolor{KLMblue}{HTML}{0FC0FC}

\definecolor{cinnamon}{rgb}{0.82, 0.41, 0.12}

\definecolor{darkorange}{rgb}{1.0, 0.55, 0.0}

\definecolor{darktangerine}{rgb}{1.0, 0.66, 0.07}

\definecolor{deepcarrotorange}{rgb}{0.91, 0.41, 0.17}

\definecolor{internationalorange}{HTML}{FF4F00}

\definecolor{persimmon}{HTML}{EC5800}

\definecolor{pumpkin}{HTML}{FF7518}

\definecolor{darkred}{rgb}{1,0,0} 
\definecolor{darkgreen}{rgb}{0,0.7,0}
\definecolor{darkblue}{rgb}{0,0,1}

\hypersetup{colorlinks,
linkcolor=darkblue,
filecolor=darkgreen,
urlcolor=darkred,
citecolor=darkgreen}

\makeatletter
\def\reflb#1#2{\begingroup
    #2%
    \def\@currentlabel{#2}%
    \phantomsection\label{#1}\endgroup
}
\makeatother

\newcommand{\labell}[1] {\label{#1}}

\numberwithin{equation}{section}
\newtheorem{Theorem}{Theorem}
\numberwithin{Theorem}{section}

\newtheorem   {Lemma}[Theorem]{Lemma}

\newtheorem   {Proposition}[Theorem]{Proposition}
\newtheorem   {Corollary}[Theorem]{Corollary}
\theoremstyle {definition}
\newtheorem   {Definition}[Theorem]{Definition}
\theoremstyle {remark}
\newtheorem   {Remark}[Theorem]{Remark}

\def    \eps    {\epsilon}

\newcommand{\CA}{{\mathcal A}}
\newcommand{\CC}{{\mathcal C}}

\newcommand{\CS}{{\mathcal S}}

\newcommand{\Ham}{{\mathit{Ham}}}

\newcommand{\id}{{\mathit id}}

\newcommand{\tL}{\tilde{L}}

\newcommand{\CB}{{\mathcal B}}

\def    \F      {{\mathbb F}}

\def    \R      {{\mathbb R}}

\def    \Z      {{\mathbb Z}}
\def    \N      {{\mathbb N}}

\def    \T      {{\mathbb T}}
\def    \CP     {{\mathbb C}{\mathbb P}}

\def    \12     {{\frac{1}{2}}}

\def    \p      {\partial}

\def    \HF     {\operatorname{HF}}

\def    \H      {\operatorname{H}}

\def    \CF      {\operatorname{CF}}

\def    \vol     {\operatorname{vol}}

\def    \s     {\operatorname{c}}

\def    \hn    {\scriptscriptstyle{H}}
\def    \Fl    {\scriptscriptstyle{Fl}}
\def    \inv   {\mathrm{inv}}

\newcommand \diam   {\operatorname{diam}}

\newcommand    \htop  {\operatorname{h_{\scriptscriptstyle{top}}}}
\newcommand   \hbr {\operatorname{\hbar}}
\newcommand   \hhbr {\operatorname{\hat{\hbar}}}
\newcommand   \hhbar {\operatorname{\hat{\hbar}}}

\newcommand   \Cr   {\scriptscriptstyle{CR}}

\begin{document}


\setlength{\smallskipamount}{6pt}
\setlength{\medskipamount}{10pt}
\setlength{\bigskipamount}{16pt}





\title [On the Growth of the Floer Barcode]{On the Growth of the Floer
  Barcode}

\author[Erman \c C\. inel\. i]{Erman \c C\. inel\. i}
\author[Viktor Ginzburg]{Viktor L. Ginzburg}
\author[Ba\c sak G\"urel]{Ba\c sak Z. G\"urel}

\address{E\c C: Institut de Math\'ematiques de Jussieu - Paris Rive
  Gauche (IMJ-PRG), 4 place Jussieu, Boite Courrier 247, 75252 Paris
  Cedex 5, France} \email{erman.cineli@imj-prg.fr}

\address{VG: Department of Mathematics, UC Santa Cruz, Santa
  Cruz, CA 95064, USA} \email{ginzburg@ucsc.edu}

\address{BG: Department of Mathematics, University of Central Florida,
  Orlando, FL 32816, USA} \email{basak.gurel@ucf.edu}



\date{\today} 

\thanks{The work is partially supported by NSF CAREER award
  DMS-1454342 (BG), Simons Foundation Collaboration Grants 581382 (VG)
  and 855299 (BG) and ERC Starting Grant 851701 via a postdoctoral
  fellowship (E\c{C})}



\begin{abstract}
  This paper is a follow up to the authors' recent work on barcode
  entropy. We study the growth of the barcode of the Floer complex for
  the iterates of a compactly supported Hamiltonian diffeomorphism. In
  particular, we introduce sequential barcode entropy which has
  properties similar to barcode entropy, bounds it from above and is
  more sensitive to the barcode growth. In the same vein, we explore
  another variant of barcode entropy based on the total persistence
  growth and revisit the relation between the growth of periodic
  orbits and topological entropy. We also study the behavior of the
  spectral norm, aka the $\gamma$-norm, under iterations. We show that the
  $\gamma$-norm of the iterates is separated from zero when the map
  has sufficiently many hyperbolic periodic points and, as a
  consequence, it is separated from zero $C^\infty$-generically in
  dimension two. We also touch upon properties of the barcode entropy
  of pseudo-rotations and, more generally, $\gamma$-almost periodic
  maps.
\end{abstract}

\maketitle

\tableofcontents

\section{Introduction}
\label{sec:intro}
The main theme of the paper is the growth of the Floer complex for the
iterates of a compactly supported Hamiltonian diffeomorphism. In
particular, we focus on the exponential growth rate of the barcode of
the Floer complex and the behavior of the spectral norm.

There is a wide range of interpretations of the question of the growth
of the Floer complex for the iterates $\varphi^k$ and there seems to
be no obvious way to make the question precise readily fitting all or
most of the aspects of the problem. When trying to articulate this
question, it is useful to keep two facts in mind.

First of all, in most cases the ``effective'' diameter of the action
spectrum $\CS\big(\varphi^k\big)$ grows at most polynomially with the
order of iteration. For instance, as a simple consequence of the
isoperimetric inequality, the diameter
\[
\diam \CS\big(\varphi^k\big)
:=\max\CS\big(\varphi^k\big)-\min\CS\big(\varphi^k\big)
\]
of this set for contractible orbits grows at most linearly in $k$ when
the underlying symplectic manifold $M$ is a surface of genus $g\geq 2$
and it grows at most quadratically when $M=\T^{2n}$; cf.\
\cite{Po}. In fact, we are not aware of any example where the diameter
would be shown to grow faster than linearly. When $M$ is not
aspherical or non-contractible periodic orbits are included, the
notion of the diameter is more involved and less unambiguous, but even
then the effective diameter grows polynomially in most cases; see
\cite[Rmk.\ 3.4]{CGG:Entropy}. Similarly, the index spectrum defined
as in, e.g., \cite{GG:gaps} can be shown to grow at most polynomially
in many situations.

On the other hand, the number of $k$-periodic points of $\varphi$ can
grow arbitrarily fast, e.g., superexponentially. Moreover, in
dimension two this behavior is in some sense common, i.e., occurs for
a dense subset of an open set of Hamiltonian diffeomorphisms in the
analytic topology; see \cite{As} and references therein.

With this in mind, one useful way to measure the growth of the Floer
complex of $\varphi^k$ is by counting the number of bars
$b_\eps\big(\varphi^k\big)$ of length greater than $\eps>0$ in its
barcode. The limit $\hbar(\varphi)$ as $\eps\searrow 0$ of the
exponential growth rate of $b_\eps\big(\varphi^k\big)$ is called the
\emph{barcode entropy} and closely related to the topological entropy
$\htop(\varphi)$ of $\varphi$; \cite{CGG:Entropy}. In particular,
$\hbar(\varphi)\leq \htop(\varphi)$ and hence
$b_\eps\big(\varphi^k\big)$ grows at most exponentially in $k$, and
$\hbar(\varphi)=\htop(\varphi)$ in dimension two.

Replacing the fixed treshold $\eps$ by $\eps_k>0$ which depends on $k$
gives rise to a somewhat different way to measure the size of the
Floer complex, which we explore in this paper. The exponential growth
rate of $b_{\eps_k}\big(\varphi^k\big)$, where $\eps_k\searrow 0$
subexponentially, is encoded by the \emph{sequential barcode entropy}
$\hhbar(\varphi)\geq \hbar(\varphi)$ of $\varphi$. We show that
sequential barcode entropy has properties similar to barcode
entropy. For instance, we still have that
$\hhbar(\varphi)\leq \htop(\varphi)$, and hence
$b_{\eps_k}\big(\varphi^k\big)$ grows at most exponentially when
$\eps_k$ is subexponential, and again
$\hhbar(\varphi)=\hbar(\varphi)=\htop(\varphi)$ in dimension
two. (Hypothetically it is possible that
$\hhbar(\varphi)=\hbar(\varphi)$  in all dimensions. However,
neither $\hhbar(\varphi)$ nor $\hbar(\varphi)$ is equal to
$\htop(\varphi)$ in general when $\dim M\geq 6$; see \cite{Ci}.)

A different perspective on the barcode growth and a variant of barcode
entropy is explored in Section \ref{sec:applications}. There we
introduce and prove some basic properties of a version of entropy
based on the growth of the total persistence of the barcode of
$\varphi^k$. In that section, we also revisit the problem of relating
the growth of periodic orbits and topological entropy. Namely, while
in general beyond dimension two these two invariants are known not to
be connected (see, e.g., \cite{As,Kal}) even for Hamiltonian
diffeomorphisms, we show that the exponential growth rate of periodic
orbits does give a lower bound for the topological entropy, albeit
with a correction term coming from the decay of the shortest bar.

The spectral norm, aka the $\gamma$-norm,
$\gamma(\varphi)\leq \diam \CS(\varphi)$ is roughly speaking the
difference between the homological maximum and the homological minimum
of the action functional (see, e.g., \cite{EP, Oh:gamma, Oh:constr,
  Sc, Vi}), giving another and more robust way to measure the
``diameter'' of $\CS(\varphi)$. (We will recall the definition of the
$\gamma$-norm in Section \ref{sec:gamma}.) Upper bounds on the
$\gamma$-norm have been extensively studied --- see the references
above and also, e.g., \cite{KS,Sh:V,Sh:V2}. Here we are interested in
lower bounds on the sequence $\gamma\big(\varphi^k\big)$ and more
specifically in the question if this sequence can get arbitrarily
close to zero. The connection with the barcode growth comes from the
fact that $\gamma(\varphi)$ gives an upper bound on the boundary
depth, i.e., the length of the largest finite bar,
$\beta_{\max}(\varphi)$; see \cite{Us1,Us2} and \cite{KS}. From this
perspective, very little seems to be known about the behavior of the
sequence $\gamma\big(\varphi^k\big)$ outside the case of
pseudo-rotations; see \cite{GG:PR,JS}.  Here, refining \cite[Prop.\
6.5]{CGG:Entropy}, we show that the sequence
$\gamma\big(\varphi^k\big)$ is bounded away from zero when $\varphi$
has sufficiently many hyperbolic periodic points. This is the case,
for instance, when $\htop(\varphi)>0$ and $\dim M=2$; or when
$\varphi$ is a strongly non-degenerate Hamiltonian diffeomorphism of a
positive genus surface. In particular, in dimension two, this sequence
is bounded away from zero $C^\infty$-generically; cf.\ \cite{LCS}.

Finally, we also look at another extreme and examine
\emph{$\gamma$-approximate identities}, the maps whose iterates
approximate the identity arbitrarily well with respect to the
$\gamma$-norm, and \emph{$\gamma$-almost periodic maps}, i.e., the
maps such that $\gamma\big(\varphi^k\big)$ becomes arbitrarily small
with positive frequency; see \cite{GG:AI}. The main, and to date the
only, source of such maps are Hamiltonian pseudo-rotations; cf.\
\cite{GG:PR,JS}. In particular, we show that for a $\gamma$-almost
periodic map $\varphi$ the sequence $b_\eps\big(\varphi^k\big)$ is
bounded for all $\eps >0$ and hence $\hbar(\varphi)=0$. This is a step
toward the proof of the conjecture that $\htop(\varphi)=0$ for
Hamiltonian pseudo-rotations.

This paper is formally independent from \cite{CGG:Entropy}, but
conceptually it is a follow up to that work and perhaps should be
better read with that work in mind.

\medskip\noindent{\bf Acknowledgements.} We are grateful to David
Burguet for useful
discussions.

\section{Definitions and results}
\label{sec:results} 
\subsection{Floer homology and barcode entropy}
\label{sec:def}

Throughout the paper we use conventions and notation from
\cite{CGG:Entropy}. Referring the reader to \cite[Sect.\
3]{CGG:Entropy} and references therein for a much more detailed
discussion, here we only touch upon several key points.

\subsubsection{Floer homology and barcodes}
In this paper all Lagrangian submanifolds $L \subset M$ are assumed to
be closed and monotone with minimal Chern number at least 2.
Hamiltonian diffeomorphisms $\varphi$ are always required to have
compact support, and when $M$ is not compact we assume that it is
sufficiently well-behaved at infinity (e.g., convex, or wide in the
sense of \cite[Defn.\ 3.1]{Gu}) so that the filtered Floer complex and
homology for the pair $(L,\varphi(L))$ or the map $\varphi$ itself can
be defined; cf.\ \cite[Rmk.\ 2.8]{CGG:Entropy}.

For the sake of simplicity Floer complexes and homology and also the
ordinary homology are taken over the ground field $\F=\F_2$. When $L$
and $L'$ are Hamiltonian isotopic and intersect transversely, we
denote by $\CF(L,L')$ the Floer complex of the pair $(L,L')$. This
complex is generated by the intersections $L\cap L'$ over the
\emph{universal Novikov field} $\Lambda$. This is the field of formal
sums
$$
\lambda=\sum_{j\geq 0} f_j T^{a_j},
$$
where $f_j\in \F$ and $a_j\in\R$ and the sequence $a_j$ (with
$f_j\neq 0$) is either finite or $a_j\to\infty$.

Due to our choice of the Novikov field, the complex $\CF(L,L')$ is not
graded. However, fixing a Hamiltonian isotopy from $L$ to $L'$ and
``cappings'' of intersections, we obtain a filtration on $\CF(L,L')$ by
the Hamiltonian action. The differential on the complex is defined in
the standard way. Note that the complex breaks down into a direct sum
of subcomplexes over homotopy classes of paths from $L$ to $L'$. Then,
to define the action filtration on $\CF(L,L')$, we also need to pick a
reference path in every homotopy class.

The barcode $\CB(L, L')$ of the Floer complex $\CF(L, L')$, in the
most refined form, is a collection of finite or semi-infinite
intervals defined in general up to some shift ambiguity. For our
purposes, it is convenient to forgo the location of the intervals and
treat $\CB(L, L')$ as a collection (i.e., a multiset) of positive
numbers including $\infty$. A construction of barcodes suitable for
our purposes is introduced and worked out in detail in \cite{UZ} and
also discussed in \cite{CGG:Entropy}. Below we briefly go over it.

Set $\CC=\CF(L, L')$ and fix a filtration
$\CA \colon \CC \to \R \cup \{-\infty\}$ via Hamiltonian action where
$\CA(0)$ is set to $-\infty$. A finite set of vectors $\xi_i\in \CC$
is called \emph{orthogonal} if for any collection
$\lambda_i\in\Lambda$ we have
$\CA (\sum\lambda_i\xi_i )=\max \CA(\lambda_i\xi_i)$. A
$\Lambda$-basis $\{\alpha_i, \,\eta_j, \,\gamma_j\}$ of $\CC$ is
called a \emph{singular decomposition} if it is orthogonal and
$\p_{\Fl}\alpha_i=0$, $\p_{\Fl}\gamma_j=\eta_j$. It is shown in
\cite[Sects. 2 and 3]{UZ} that $\CC$ admits a singular
decomposition. Ordering the pairs $(\eta_j,\,\gamma_j)$ by the action
difference, we obtain
$$
\CA(\gamma_1)-\CA(\eta_1)\leq\CA(\gamma_2)-\CA(\eta_2)\leq\ldots.
$$
In this paper we refer to the multiset formed by the differences
$\CA(\gamma_i)-\CA(\eta_i)$ together with $\dim_\Lambda\HF(L,L')$ many
$\infty$'s (corresponding to the basis elements $\alpha_i$) as the
\emph{barcode} of $\CC=\CF(L, L')$ and denote it by
$\CB(L,L')$. Moreover, abusing notation, we call these numbers
\emph{finite/infinite bars}. In the original definition \cite[Def.\
6.3]{UZ}, barcode also contains information about the location of
these bars, i.e., the bars are pinned, whereas our version
only keeps the length of the bars. For our purposes the length data
suffices; hence, for simplicity, we forgo the locations. The barcode
$\CB(L,L')$ is independent of the choice of a singular decomposition
and other auxiliary data involved in the construction of $\CF(L, L')$;
see \cite[Thm.\ 7.1]{UZ} and \cite[Prop.\ 6.2]{Us2}. Also note that
$\CB(L,L')=\CB(L',L)$ as was shown in \cite[Prop.\ 2.20]{UZ}; see also
\cite[Sec.\ 3.3.1]{CGG:Entropy}.

Recall that the \emph{Hofer norm} of a Hamiltonian diffeomorphism
$\varphi \colon M \to M$ is defined as
$$
\|\varphi\|_{\hn} = \inf_{H}\int_{S^1}\big(\max_M H_t-\min_M
H_t\big)\, dt,
$$
where the infimum is taken over all 1-periodic in time Hamiltonians
$H$ generating $\varphi$, i.e., $\varphi=\varphi_H$. We refer the
reader to, e.g., \cite{Po:Book} and references therein for a very
detailed discussion of the Hofer norm. The \emph{Hofer distance}
between two Hamiltonian isotopic Lagrangian submanifolds $L$ and $L'$
is
\[
d_{\hn}(L,L')=\inf\big\{\|\varphi\|_{\hn}\mid \varphi(L)=L'\big\};
\]
see \cite{Ch} and, e.g., \cite{Us} for further references.

The most important feature of the barcode, at least for our purposes,
is that it is continuous in the Lagrangian with respect to the
$C^\infty$-norm and even the Hofer norm (or the $\gamma$-norm; see
Section \ref{sec:gamma}). In order to state the continuity property,
let us first recall the relevant metric on the space of barcodes. We
say that the barcodes $\CB_1$, $\CB_2$ are \emph{$\delta$-matched} if,
after deleting as needed bars of length $<2\delta$, one can find a
bijection $\CB_1 \to \CB_2$; $\beta_1^i \mapsto \beta_2^i$ such that
$\vert \beta_1^i - \beta_2^i \vert < 2\delta$. The infimum of all such
$\delta$'s is called the \emph{bottleneck distance} between
$\CB_i$. This is indeed a distance on the space of un-pinned barcodes,
which is bounded from above by the bottleneck distance on the space of
pinned barcodes under the natural forgetful map between these spaces.
The continuity property can be stated as follows; see
\cite[Sect.\ 12]{UZ} for details. Assume that Lagrangian submanifolds
$L$, $L'$ and $L''$ are Hamiltonian isotopic such that
$L \pitchfork L'$ and $L \pitchfork L''$. Then the bottleneck distance
between $\CB(L, L')$ and $\CB(L, L'')$ is bounded from above by the
Hofer distance $d_{\hn}(L',L'')$. This property allows one to extend
the definition of the barcode ``by continuity'' to the case where the
manifolds are not transverse.

\subsubsection{Barcode entropy} In this section we review the
definition of barcode entropy introduced in \cite{CGG:Entropy}. Let
$L$, $L'$ be two transverse Lagrangians as in the previous section and
let $\CB(L, L')$ be the barcode of the Floer complex $\CF(L, L')$. Set
\[
  b_\eps(L,L'):= \big|\{\beta\in \CB(L, L') \mid \beta>\eps\}\big|,
  \]
  and denote the total number of bars in the barcode by 
\[
  b(L,L'):=|\CB(L,L')|\geq b_\eps(L,L').
\]
Omitting the definition of the barcode in the non-transverse case, we
extend the barcode counting function $b_\eps(L,L')$ to the situation
where $L$ and $L'$ need not be transverse by setting
\begin{equation}
  \label{eq:b-eps3}
b_\eps(L,L'):=\liminf_{\tL\to L'}b_\eps(L,\tL)\in \Z.
\end{equation}
Here the limit is taken over all Lagrangian submanifolds $\tL\pitchfork L$
which are Hamiltonian isotopic to $L'$ and converge to $L'$ in the
$C^\infty$-topology (or at least in the $C^1$-topology).  As a
consequence, $d_{\hn}(\tL,L')\to 0$, where $d_{\hn}$ is the Hofer
distance. Alternatively, we could have required $\tL$ be Hamiltonian
isotopic to $L$, transverse to $L'$ and converge to $L$.  Since
$b_\eps(L,L')\in\Z$, the limit in \eqref{eq:b-eps3} is necessarily
attained, i.e., there exists $\tL$ arbitrarily close to $L'$ such that
$b_\eps(L,L')=b_\eps(L,\tL)$. Observe that definition
\eqref{eq:b-eps3} extends to the transverse case. Namely, one
direction is a consequence of the ``$C^\infty$-stability'' of
essentially all the data related to $\CF(L, L')$; alternatively,
though unnecessary, one can use the continuity of barcodes. As for the
other direction, one can, for instance, take the constant sequence.

\begin{Remark}
  In this paper we use the barcode counting function as a stable lower
  bound for the number of intersections. Namely, first of all, note
  that
$$
|L\cap L'|=\dim_\Lambda\CF(L,L')=2b(L,L')-\dim_\Lambda\HF(L,L')\geq
b(L,L')\geq b_\eps(L,L')
$$
whenever $L \pitchfork L'$. This directly follows from the definition
of the barcode. In particular, in the transverse case, $b_\eps(L,L')$
gives a lower bound for the number of intersections:
\begin{equation}
\label{eq:intersections-b}
|L\cap L'| \geq b(L,L')\geq b_\eps(L,L').
\end{equation}
Assume now that Lagrangians $L$, $L'$ and $L''$ are Hamiltonian
isotopic, $L''\pitchfork L$ and $d_{\hn}(L',L'')<\delta/2$. Then,
whether or not $L$ and $L'$ are transverse, we have
 \begin{equation}
   \label{eq:intersections}
  |L\cap L''|\geq b_{\eps}(L,L'')\geq b_{\eps+\delta}(L,L').
\end{equation}
Here the first inequality is just \eqref{eq:intersections-b} and the
second inequality, which holds regardless of $L\pitchfork L''$ or not,
is the extension of the continuity property from the previous section
via the limit \eqref{eq:b-eps3} to the non-transverse case.
\end{Remark}

\begin{Definition}[Relative Barcode Entropy]
  \label{def:hbr-rel12}
  The \emph{barcode entropy of $\varphi$ relative to $(L,L')$} is
  $$
  \hbr(\varphi;L,L'):=\lim_{\eps\searrow 0} \hbr_\eps (\varphi;
  L,L')\in [0, \infty],
  $$  
  where
  $$
  \hbr_\eps(\varphi; L,L'):=\limsup_{k\to \infty}\frac{\log^+
    b_\eps\big(L,L^k\big)}{k}\textrm{ and }  L^k:=\varphi^k(L').
  $$
 \end{Definition}
 Here and throughout the paper the logarithm is taken base 2 and
 $\log^+=\log$ except that $\log^+0=0$. Note that
 $\hbr_\eps(\varphi; L,L')$ is increasing as $\eps\searrow 0$, and
 hence the limit exists, although \emph{a priori} it could be
 infinite.

 Next, we discuss the absolute barcode entropy. Let $M$ be a closed
 monotone symplectic manifold and again let $\varphi\colon M\to M$ be
 a Hamiltonian diffeomorphism. Then we can apply the above
 constructions to $L=\Delta=L'$, the diagonal in the symplectic square
 $\big(M\times M, (-\omega,\omega)\big)$, with $\varphi$ replaced by
 $\id\times \varphi$, or directly to the Floer complex $\CF(\varphi)$
 of $\varphi$ \emph{for all free homotopy classes of loops in $M$}. In
 the latter case we denote by $\CB(\varphi)$ the resulting
 barcode. For instance, we have
\begin{align*}
  b_\eps\big(\varphi^k\big)
  &:=b_\eps\big(L,L^k\big)\\
  &=\big|\{\textrm{bars of length greater than
    $\eps$ in the barcode $\CB\big(\varphi^k\big)$}\}\big|,
\end{align*}
where $L=\Delta$ and $L^k$ is the graph of $\varphi^k$, and in the
second equality we tacitly assumed that $\varphi^k$ is
non-degenerate. We emphasize that we include the 1-periodic orbits in
all free homotopy classes of loops in $M$ as generators of
$\CF(\varphi)$ in contrast with a more common definition involving
only contractible 1-periodic orbits. This is absolutely essential for
the definition of barcode entropy.

\begin{Definition}[Absolute Barcode Entropy]
  \label{def:hbr-bar}
  The \emph{$\eps$-barcode entropy} of $\varphi$ is
  $$
  \hbr_\eps(\varphi):=\limsup_{k\to \infty}\frac{\log^+
    b_\eps\big(\varphi^k\big)}{k}
  $$
  and the \emph{(absolute) barcode entropy} of $\varphi$ is
  $$
  \hbr(\varphi):=\lim_{\eps\searrow 0} \hbr_\eps(\varphi) \in
  [0,\,\infty]
  $$
  or, in other words,
  $$
  \hbr(\varphi):=\hbr(\id\times\varphi;\Delta, \Delta).
  $$
\end{Definition}
Again, the limit in the definition of $\hbr(\varphi)$ exists since
$\hbr_\eps(\varphi)$ is increasing as $\eps\searrow 0$.

By \cite[Thm.\ 5.1]{CGG:Entropy},
 $\hbar(\varphi; L,L')\leq \htop(\varphi)<\infty$, and hence
 $b_\eps(L,L^k)$ grows at most exponentially:
 \begin{equation}
   \label{eq:growth-L}
    b_\eps(L,L^k)\leq 2^{c k} \textrm{ for large } k,
  \end{equation}
  where we can take any $c>\htop(\varphi)$. In particular,
  $\hbar(\varphi)\leq \htop(\varphi)<\infty$ and again
    \begin{equation}
      \label{eq:growth-phi}
      b_\eps\big(\varphi^k\big)\leq 2^{c k}
    \end{equation}
    for large $k$. At the same time, the number of periodic points of
    $\varphi$, and hence $b\big(\varphi^k\big)$, can grow arbitrarily
    fast and, as a consequence, the shortest bar can also go to zero
    arbitrarily fast; see \cite{As} and Section
    \ref{sec:applications}. One of our goals in this paper is to
    refine \eqref{eq:growth-L}, \eqref{eq:growth-phi} and \cite[Thm.\
    5.1]{CGG:Entropy}.

    \begin{Remark}
      Although Definition \ref{def:hbr-bar} closely resembles the
      definition of topological entropy, the similarity is rather
      deceiving. For instance, the $\eps$-entropy family
      $\hbar_\eps(\varphi)$ is well-defined for every $\eps>0$, while
      its counterpart for topological entropy depends on the
      background metric. Note that as a consequence when
      $\hbar(\varphi)>0$ we obtain a new numerical invariant of
      $\varphi$: the treshold value of $\eps$ for which the entropy is
      positive, $\sup\{\eps>0\mid \hbar_\eps(\varphi)>0\}$.
    \end{Remark}

\subsection{Barcode growth and sequential entropy}
\label{sec:growth}
In this section we consider the growth of $b_{\eps_k}\big(L,L^k\big)$
for a certain class of sequences $\eps_k>0$. To be more precise, a
bounded sequence $\eps_k>0$ is said to be \emph{subexponential} if
$$
\eps_k 2^{\eta k}\to\infty \textrm{ for all } \eta>0
$$
or, equivalently,
\begin{equation}
  \label{eq:subexp2}
\lim_{k\to\infty}\frac{\log^+ \eps_k }{k}=0.
\end{equation}
For instance, a constant sequence or a polynomially decaying sequence
is subexponential. For a sequence $\eps_k>0$, define the
\emph{relative sequential $\{\eps_k\}$-barcode entropy} to be
  $$
  \hhbr_{\{\eps_k\}}(\varphi; L,L'):=\limsup_{k\to \infty}\frac{\log^+
    b_{\eps_k}\big(L,L^k\big)}{k}\in [0, \infty]\textrm{ with }
  L^k:=\varphi^k(L').
  $$
  Furthermore, let us partially order positive sequences by
  $\{\eps'_k\}\preceq \{\eps_k\}$ whenever $\eps'_k\leq \eps_k$ for
  all large $k\in\N$. Clearly,
  \begin{equation}
    \label{eq:order}
  \hhbr_{\{\eps'_k\}}(\varphi; L,L')\geq \hhbr_{\{\eps_k\}}(\varphi;
  L,L')
  \textrm{ when } \{\eps'_k\}\preceq \{\eps_k\}.
\end{equation}

Then, in the setting of Section \ref{sec:def}, we define the
\emph{relative sequential barcode entropy} as
  $$
  \hhbr(\varphi;L,L'):=\sup_{\{\eps_k\}}\hhbr_{\{\eps_k\}} (\varphi;
  L,L')\in [0, \infty],
  $$
  where the supremum is taken over all subexponential sequences
  $\{\eps_k\}$. (In this definition one can replace the supremum by
  the direct limit with respect to the reversed partial ordering,
  i.e., as the sequences get closer and closer to zero.) The absolute
  sequential entropy $\hhbr(\varphi)$ and its $\{\eps_k\}$-counterpart
  $\hhbr_{\{\eps_k\}}(\varphi)$ are defined in a similar fashion. By
  \eqref{eq:order}, for every subexponential sequence $\eps_k\to 0$,
  we have
  \begin{equation}
    \label{eq:entropies}
    \hbr(\varphi;L,L')\leq
    \hhbr_{\{\eps_k\}}(\varphi;L,L')\leq\hhbr(\varphi;L,L')\textrm{ and }
    \hbr(\varphi)\leq \hhbr_{\{\eps_k\}}(\varphi)\leq\hhbr(\varphi).
  \end{equation}
  (We are not aware of any examples where the inequality between
  $\hhbar$ and $\hbar$ is strict and hypothetically it is possible
  that the sequential barcode entropy is always equal to the barcode
  entropy.)  Our key result is the following refinement of
  \eqref{eq:growth-L}, \eqref{eq:growth-phi} and \cite[Thm.\
  5.1]{CGG:Entropy}.

\begin{Theorem}
  \label{thm:growth}
  Let $\{\eps_k\}$ be a subexponential sequence. Then
  $b_{\eps_k}\big(L,L^k\big)$ grows at most exponentially.
  Furthermore,
$$
\hhbr(\varphi;L,L')\leq\htop (\varphi).
$$
\end{Theorem}

Note that, as a consequence,
$\hhbr_{\{\eps_k\}}(\varphi; L,L')<\infty$ and
$\hhbr(\varphi; L,L')<\infty$ which is \emph{a priori} not obvious. We
prove Theorem \ref{thm:growth} in Section \ref{sec:proofs}. Here we
only point out that the proof of Theorem \ref{thm:growth} ultimately
relies on Yomdin's theorem, \cite{Yo}, and in the theorem and
throughout the paper all maps and submanifolds are assumed to be
$C^\infty$-smooth.

Applying Theorem \ref{thm:growth} to $\id\times\varphi$ and the
diagonal, we arrive at a refinement of \cite[Thm.\ A]{CGG:Entropy}:

\begin{Corollary}
  \label{cor:growth}
  Let $\{\eps_k\}$ be a subexponential sequence. Then
  $b_{\eps_k}\big(\varphi^k\big)$ grows at most exponentially, and
$$
\hhbr(\varphi)\leq\htop (\varphi).
$$
\end{Corollary}

As a consequence of Theorem \ref{thm:growth} and Corollary
\ref{cor:growth}, $b_{\eps_k}(L,L^k)$ and
$b_{\eps_k}\big(\varphi^k\big)$ grow at most exponentially whenever
the sequence $\{\eps_k\}$ is subexponential:
$$
    b_{\eps_k}(L,L^k)\leq 2^{c k} \textrm{ and }
    b_{\eps_k}\big(\varphi^k\big)\leq 2^{c k}
    \textrm{ for large } k, 
$$
where we can take any $c>\htop(\varphi)$. These inequalities refine
\eqref{eq:growth-L} and \eqref{eq:growth-phi}.

Recall that $\hbr(\varphi)\geq \htop(\varphi|_K)$ for any (closed)
hyperbolic subset $K$; \cite[Thm.\ B]{CGG:Entropy}.  By
\eqref{eq:entropies}, this lower bound, as any lower bound on
$\hbr(\varphi)$, also holds for $\hhbr(\varphi)$.  Furthermore, by
\cite[Thm.\ C]{CGG:Entropy}, $\hbr(\varphi)=\htop (\varphi)$ when $M$
is a surface, and thus we have the following result.

\begin{Corollary}
  \label{cor:2D}
Assume that $\varphi$ is a compactly supported Hamiltonian
diffeomorphism of a surface. Then 
\begin{equation}
  \label{eq:all-entropies}
\hbr(\varphi)=\hhbr_{\{\eps_k\}}(\varphi)
=\hhbr(\varphi)=\htop(\varphi),
\end{equation}
whenever $\{\eps_k\}$ is subexponential and $\eps_k\to 0$.
\end{Corollary}

To summarize, sequential barcode entropy has essentially the same key
properties as the barcode entropy originally defined in
\cite{CGG:Entropy}. For instance, properties (i)-(iv) from Prop.\ 4.4
therein holds for sequential barcode entropy too. A possible exception
is the Hofer lower semi-continuity of the relative barcode entropy in
the Lagrangian (part (v) of Prop.\ 4.4). Unlike the properties
(i)--(iv), the proof of (v) from \cite{CGG:Entropy} does not carry
over to the sequential case. On the other hand, as we mentioned above,
it is entirely possible that in general the two entropies are
equal. Yet, Corollary \ref{cor:2D} in its full form certainly does not
generalize to higher dimensions. Namely, there are examples of
Hamiltonian diffeomorphisms where $\dim M\geq 6$ and the last equality
in \eqref{eq:all-entropies} turn into a strict inequality with
$\htop(\varphi)>0$ and
$\hbr(\varphi)=\hhbr_{\{\eps_k\}}(\varphi) =\hhbr(\varphi)=0$; see
\cite{Ci}.

\begin{Remark}[Topological sequential entropy]
  One could also modify the definition of topological entropy in a way
  similar to sequential barcode entropy. The resulting ``sequential''
  topological entropy is equal to the topological entropy for
  $C^\infty$-maps of compact manifolds $M$. This is a consequence of
  Yomdin's theory \cite[Prop. 3.10]{Bu}.  (We are grateful to
    David Burguet for explaining to us the connection and a proof of
    the equality.)  On the other hand, it is not hard to construct a
  $C^0$-map with zero topological entropy and, for instance, infinite
  ``sequential'' topological entropy.
  
  Namely, let $(M, g)$ be a closed Riemannian manifold of
  $\dim M \geq 2$ and $B_j \subset M$ be a sequence of
  disjoint balls of subexponentially decreasing radius
 $\delta_j \to 0$. We can take any such sequence
    $\delta_j$. For each $j \in \N$, there is a diffeomorphism
  $\varphi_j$ of $M$ (which can be taken to be a Hamiltonian
  diffeomorphism if the manifold is symplectic) supported in $B_j$
  with $\htop(\varphi_j) =0$ and such that the maximal number of
  $\delta_j/j$-separated points with respect to the metric
  $d_k^{\varphi_j} (x,y) := \max_{0 \leq i \leq k-1} d_g
  (\varphi_j^i(x), \varphi_j^i(y))$ is greater than $j^j$
  for $k=j$, and hence all $k\geq j$. To construct such
    a diffeomorphism, we first show that there exist $j^j$ disjoint
  finite sequences of $j$ points $x_0^s,\ldots, x_{j-1}^s$,
    $s=1,\dots, j^j$, in $B_j$ which are $\delta_j/j$-separated,
  i.e., $\max_id_g \big(x_i^s,x_i^{s'}\big)>\delta_j/j$
    whenever, $s\neq s'$. Then we define $\varphi_j\colon B_j\to B_j$
  on disjoint path-connected neighborhoods of these sequences,
    turning the sequences into orbits and making sure that
    the resulting map $\varphi_j$ has zero topological
    entropy. Let $\varphi \colon M \to M$ be the map given
    by composing all $\varphi_j$ or, equivalently, taking
    their ``disjoint union''. By construction, the ``sequential''
  topological entropy of $\varphi$ is infinite but
  $\htop(\varphi)=0$. (If, instead of $j^j$, we took $2^j$
  sequences, we would get a map with positive, but possibly
  finite, ``sequential'' topological entropy and zero topological
  entropy.)

\end{Remark}

\begin{Remark}[Lower bounds on the growth of
  $b_\eps\big(\varphi^k\big)$]
  \label{rmk:b-eps-growth}
  A related question is that of a lower bound on
  $b_\eps\big(\varphi^k\big)$ or $b_\eps\big(L,L^k\big)$, although it
  is not entirely clear how to pose this question in a meaningful
  way. The difficulty is that one cannot expect any particular growth
  behavior without additional conditions on $\varphi$ or $M$. For
  instance, obviously $b_\eps\big(\varphi^k\big)=n+1$ for all $\eps>0$
  when $\varphi$ is a non-degenerate pseudo-rotation of $\CP^n$; see
  also Proposition \ref{prop:AP}. We refer the reader to Section
  \ref{sec:AI+PR} for more details on pseudo-rotations.  Furthermore,
  when $\dim M=2$ and $\varphi$ is autonomous, or even integrable, we
  conjecture that $b_\eps\big(\varphi^k\big)$ grows at most
  polynomially with $k$. On the other hand, in all dimensions, when
  $\varphi$ has a locally maximal hyperbolic subset $K$ with
  $\htop(\varphi|_K)>0$, the sequence
  $k\mapsto b_\eps\big(\varphi^{kN}\big)$ grows exponentially for some
  $N\in \N$ and $\eps>0$. This is a consequence of \cite[Thm.\
  18.5.6]{KH} and \cite[Prop.\ 3.8 and 6.2]{CGG:Entropy}. Therefore,
  by \cite{LCS}, $b_\eps\big(\varphi^{kN}\big)$ grows exponentially
  for $C^\infty$-generic $\varphi$ when $\dim M=2$. (To be more
  precise, the set of $\varphi$ such that this sequence grows
  exponentially for some $\eps>0$ and $N$ depending on $\varphi$ is
  $C^\infty$-residual.) When the genus of $M$ is positive we do not
  have any example of a strongly non-degenerate $\varphi$ with
  $\htop(\varphi)=0$, although we believe that such Hamiltonian
  diffeomorphisms exist. It is not clear what growth of
  $b_\eps\big(\varphi^k\big)$ one should expect in this case.
  \end{Remark}

  \subsection{Lower bounds on the $\gamma$-norm}
\label{sec:gamma}
As observed in \cite[Sect.\ 6.1.5]{CGG:Entropy} our results on barcode
and topological entropy and their proofs yield as a byproduct lower
bounds on the spectral norm a.k.a.\ $\gamma$-norm of the iterates of
$\varphi$. One of such results is Proposition \ref{prop:gamma}, stated
below, which refines \cite[Prop.\ 6.5]{CGG:Entropy}.

Recall that when $M^{2n}$ is a closed symplectic manifold the
$\gamma$-norm of a Hamiltonian diffeomorphism $\varphi$ is defined as
$$
\gamma(\varphi) = \inf\big\{\s(H)+\s\big(H^{\inv}\big)\mid \varphi
=\varphi_H\big\},
$$
where $H^{\inv}$ is the Hamiltonian generating the flow
$\big(\varphi_H^{t}\big)^{-1}$ and $\s$ is the spectral invariant
associated with the fundamental class $[M]\in \H_{2n}(M)$. Then
$$
\gamma(\varphi)\leq \|\varphi\|_{\hn}.
$$
We refer the reader to, e.g., \cite{EP, Oh:gamma, Oh:constr, Sc, Vi}
for the proof and a further discussion of the $\gamma$-norm.

\begin{Proposition}
  \label{prop:gamma}
  Let $\varphi\colon M\to M$ be a Hamiltonian diffeomorphism of a
  closed weakly monotone symplectic manifold with more than
  $\dim \H_*(M)$ hyperbolic periodic points.  Then the sequence
  $\gamma\big(\varphi^k\big)$, $k\in\N$, is bounded away from zero.
\end{Proposition}

We prove this proposition in Section \ref{sec:pf-prop-gamma}. Note
that having more than $\dim \H_*(M)$ hyperbolic periodic points, or
more than any fixed number of hyperbolic periodic points, is an open
property in $C^1$-topology. By the Conley conjecture \cite{FH, Gi:CC,
  SZ} every Hamiltonian diffeomorphism $\varphi$ of a positive genus
surface $\Sigma_{g\geq 1}$ has infinitely many periodic points. By the
Lefschetz formula, roughly speaking, at least half of these periodic
points are hyperbolic if $\varphi$ is strongly non-degenerate which is
the case $C^\infty$-generically (recall that, in dimension two, an
elliptic or a negative hyperbolic fix point has Lefschetz index 1 and
a positive hyperbolic one has index $-1$). Hence we have:
\begin{Corollary}
\label{cor:conley}
There exists a $C^1$-open neighborhood
$U \subset \Ham(\Sigma_{g \geq 1}, \omega)$ of the set of strongly
non-degenerate Hamiltonian diffeomorphisms of $\Sigma_{g \geq 1}$ such
that the sequence $\gamma\big(\varphi^k\big)$, $k\in\N$, is bounded
away from zero for every $\varphi \in U$. In particular, this sequence
is bounded away from zero $C^\infty$-generically.
\end{Corollary} 

\begin{Remark} 
\label{rmk:sphere}
For a strongly non-degenerate Hamiltonian diffeomorphism $\varphi$ of
$M=S^2$, the sequence $\gamma\big(\varphi^k\big)$ is not bounded away
from zero if and only if $\varphi$ is a pseudo-rotation. Indeed, this
sequence contains a subsequence converging to zero for all, not
necessarily non-degenerate, pseudo-rotations of $\CP^n$; see
\cite{GG:PR}.  In the opposite direction, when $M=S^2$, the existence
of one hyperbolic periodic point is enough to bound the sequence
$\gamma(\varphi^k)$ away from zero. Strictly speaking, we need a
positive hyperbolic periodic point but the positivity can always be
achieved up to passing to an even iteration. Hence, more generally,
without any non-degeneracy assumption, if this is not the case for
$\varphi$, then all periodic points of $\varphi$ are elliptic (by our
conventions 1 is an elliptic eigenvalue, hence in dimension two, all
degenerate fixed points are elliptic). For strongly non-degenerate
Hamiltonian diffeomorphisms $\varphi$, by the Lefschetz formula, this
forces $\varphi$ to be a pseudo-rotation.
\end{Remark}

The conditions of the proposition are satisfied by even a wider margin
when $\varphi$ has a hyperbolic invariant set with positive
topological entropy. (We refer the reader to, e.g., \cite[Sect.\
6]{KH} for the definition and a detailed discussion of
\emph{hyperbolic} invariant sets. Here, all such sets are required to
be compact by definition.) To be more precise, recall that a compact
invariant set $K$ of $\varphi$ is said to be \emph{locally maximal} if
there exists a neighborhood $U\supset K$ such that $K$ is the maximal
invariant subset of $U$ or, in other words, $x\in K$ whenever the
entire orbit $\{\varphi^k(x)\mid k\in\Z\}$ through $x$ is contained in
$U$. For instance, the orbit $K=\{\varphi^k(x)\}$ of a hyperbolic
periodic point $x$ is locally maximal. By \cite[Thm.\ 3.3]{ACW},
whenever $\varphi$ has a hyperbolic invariant set $K$ it also has a
locally maximal hyperbolic invariant set $K'$ with
$\htop\big(\varphi|_{K'}\big)$ arbitrarily close to
$\htop\big(\varphi|_K\big)$. In particular,
$\htop\big(\varphi|_{K'}\big)>0$ if $\htop\big(\varphi|_K\big)>0$, and
$\varphi|_{K'}$ has infinitely many (hyperbolic) periodic orbits; cf.\
\cite[Thm.\ 18.5.6]{KH}. Thus we have proved the following.

\begin{Corollary}
  \label{cor:gamma1}
  Assume that a Hamiltonian diffeomorphism $\varphi\colon M\to M$ has
  a hyperbolic invariant set $K$ with
  $\htop\big(\varphi|_K\big)>0$. Then the sequence
  $\gamma\big(\varphi^k\big)$, $k\in\N$, is bounded away from zero.
\end{Corollary}

By the results from \cite{Ka}, $\varphi$ always has a hyperbolic
invariant set when $\dim M=2$ and $\htop(\varphi)>0$. Furthermore,
$C^\infty$-generically $\htop(\varphi)>0$ in dimension two as is
proved in \cite{LCS}. Therefore, we have the following (cf. Corollary
\ref{cor:conley} and Remark \ref{rmk:sphere}).

\begin{Corollary}
  \label{cor:gamma}
  Let $\varphi$ be a Hamiltonian diffeomorphism of
  a surface $M$ with $\htop(\varphi)>0$. Then the sequence
  $\gamma\big(\varphi^k\big)$, $k\in\N$, is bounded away from zero. In
  particular, again, this sequence is bounded away from zero
  $C^\infty$-generically in dimension two.
\end{Corollary}

\begin{Remark}
  Actually, to derive from Proposition \ref{prop:gamma} the fact that
  the sequence $\gamma\big(\varphi^k\big)$ is bounded away from zero
  $C^\infty$-generically in dimension two, we do not need to invoke
  results from \cite{FH, Gi:CC, SZ} or \cite{LCS}. Indeed, note that
  whenever $\varphi$ has an elliptic periodic point one can create a
  horseshoe and hence infinitely many hyperbolic periodic points by a
  $C^\infty$-small perturbation. This is a consequence of the
  Birkhoff--Lewis theorem. Then, in the extremely hypothetical
  situation where $\varphi$ does not have elliptic periodic points, a
  standard index argument shows that it must have infinitely many
  hyperbolic periodic points.  
\end{Remark}

\begin{Remark}
\label{rmk:gamma}
The converse of Corollary \ref{cor:gamma} (or Proposition
\ref{prop:gamma}) is not true in general. Even in dimension two, the
sequence $\gamma\big(\varphi^k\big)$, $k\in\N$, can be bounded away
from zero when $\htop(\varphi)=0$. It is easy to construct an
autonomous Hamiltonian diffeomorphism $\varphi$ of a surface of
positive genus with no hyperbolic periodic points and
$\gamma\big(\varphi^k\big)\to \infty$. For instance, as in \cite[Ex.\ 5.6]{Sc}, one can take
$H=\sin(2\pi \theta)$ where $\theta$ is the first angular coordinate
on $\T^2= \R^2/\Z^2$. This is of course impossible for $M=S^2$ because
the $\gamma$-norm in this case is bounded from above; see \cite{EP}
and also \cite{Sh:V,Sh:V2}. Interestingly, other than this fact and
Remark \ref{rmk:sphere} and Corollary \ref{cor:gamma} essentially
nothing seems to be known about the behavior of the $\gamma$-norm
under iterations when $M=S^2$. For instance, we do not know whether
the converse of Corollary \ref{cor:gamma} (or Proposition
\ref{prop:gamma}) is true for $M=S^2$, or even if the sequence
$\gamma\big(\varphi^k\big)$ is bounded away from zero when the
Hamiltonian is autonomous and a convex or concave function of the
latitude.
\end{Remark}

The reader can also find further results, based on \cite{CS}, on the
generic growth of the $\gamma$-norm in \cite{CGG:Spectral}.

\subsection{Approximate identities and Hamiltonian pseudo-rotations}
\label{sec:AI+PR}
In this section, looking at the results from Section \ref{sec:gamma}
from a different perspective, we focus on two classes of maps with
subexponential growth of $b_\eps$: $\gamma$-approximate identities and
Hamiltonian pseudo-rotations of $\CP^n$.

Defining approximate identities, it is useful to work in a greater
generality than needed for our purposes. Consider a class of compactly
supported diffeomorphisms $\varphi$ of a smooth manifold $M$ (e.g.,
all such diffeomorphisms or, as above, compactly supported Hamiltonian
diffeomorphisms, etc.), equipped with some metric, e.g., the $C^0$- or
$C^1$- or $C^r$-metric or the $\gamma$-metric in the Hamiltonian case
which we are interested in here. The norm $\|\varphi\|$ is by
definition the distance from $\varphi$ to the identity. Following
\cite{GG:AI}, we will call $\varphi$ a $\|\cdot\|$-\emph{approximate
  identity}, or a $\|\cdot\|$-\emph{a.i.}\ for the sake of brevity, if
$\varphi^{k_i}\to \id$ with respect to $\|\cdot\|$ for some sequence
$k_i\to\infty$. We will often suppress the norm in the notation. In
dynamics, a.i.'s are usually referred to as \emph{rigid} maps which
sometimes clashes with the same term used for structural stability.
(We believe that a confusion with approximate identities in analysis
is unlikely.)  Approximate identities have been extensively studied,
although usually from a perspective different than ours; see, e.g.,
\cite{A-Z, FK, FKr} and also \cite{GG:AI} for further references.

The definition can be refined or modified in several ways and one such
refinement is of particular interest to us. Namely, for a given
$\eps>0$, consider the iterations $\varphi^{k_i}$ such that
$$
\big\|\varphi^{k_i}\big\|<\eps.
$$
Thus $k_i=k_i(\eps)$ is a strictly increasing sequence. Then $\varphi$
is said to be \emph{$\|\cdot\|$-almost periodic} if for every $\eps>0$
the sequence $k_i$ is \emph{quasi-arithmetic}, i.e., the difference
between any two consecutive terms is bounded by a constant, possibly
depending on $\eps$. Almost periodic maps are closely related to
compact group actions on $M$: $\varphi$ is $C^0$-almost periodic if
and only if the family $\{\varphi^k\}$ is equicontinuous and thus
generates a compact abelian group of (compactly supported)
homeomorphisms,~\cite{GH}.

Almost periodicity and rigidity (the $C^0$-a.i.\ condition) impose
strong restrictions on the dynamics of $\varphi$. For instance, a
$C^0$-a.i.\ clearly cannot be topologically mixing and every point of
$M$ belongs to its $\omega$- and $\alpha$-limit sets. In particular, a
$C^0$-a.i.\ cannot have hyperbolic invariant sets. (See, e.g.,
\cite{A-Z,GG:AI} for more details, examples and references.)
Furthermore, as pointed out in \cite[p.\ 681]{A-Z}, any $C^0$-a.i.\
has zero topological entropy, although this fact is not immediately
obvious.

In the Hamiltonian setting, it is natural to consider a.i.'s and
almost periodic maps $\varphi$ with respect to the $\gamma$-norm and
we are concerned here with the effect of these conditions on the
dynamics of $\varphi$. For instance, by Corollaries \ref{cor:gamma1}
and \ref{cor:gamma}, a $\gamma$-a.i.\ $\varphi$ cannot have locally
maximal hyperbolic invariant sets with positive entropy and in
dimension two we necessarily have $\htop(\varphi)=0$. Here we focus on
barcode entropy and the barcode growth.

\begin{Proposition}
  \label{prop:AP}
  Let $\varphi$ be a $\gamma$-almost periodic Hamiltonian
  diffeomorphism of a closed monotone symplectic manifold $M$ and let
  $L$ and $L'$ be closed Lagrangian submanifolds of $M$ as in Section
  \ref{sec:def}. Then for every $\eps>0$ the sequences
  $b_\eps\big(L,L^k\big)$ and $b_\eps\big(\varphi^k\big)$ are
  bounded. In particular, $\hbar(\varphi; L,L')=0$ and
  $\hbar(\varphi)=0$.
\end{Proposition}

We prove this proposition, which is an easy consequence of a result
from \cite{KS},  in Section \ref{sec:AP-pfs} where we also show that
$\hbar(\varphi; L,L')=0$ and $\hbar(\varphi)=0$ for $\gamma$-a.i.'s
under a certain growth condition on the sequence $k_i$; see
Proposition \ref{prop:AP2}. We conjecture that this is true for all
$\gamma$-a.i.'s. 

A $C^0$-almost periodic Hamiltonian diffeomorphism or $C^0$-a.i.\ is
automatically $\gamma$-almost periodic or, respectively,
$\gamma$-a.i.\ when $M$ is symplectically aspherical, \cite{BHS}, and
also for some other classes of monotone symplectic manifolds $M$
including $\CP^n$, \cite{Sh:V}. However, we are not aware of any
example of Hamiltonian $C^0$-a.i.'s on a symplectically aspherical
manifold and hypothetically such maps do not exist. (See \cite{Po} for
the proof in the $C^1$-case and \cite{GG:AI} for a further
discussion.)

To date, the only known examples of Hamiltonian $\gamma$-a.i.'s or
$\gamma$-almost periodic Hamiltonian diffeomorphisms (beyond those
coming from torus actions) are \emph{Hamiltonian pseudo-rotations}
$\varphi$ of $\CP^n$ (see \cite{GG:PR, GG:AI} and also \cite{JS}),
although one can expect the same to be true for Hamiltonian
pseudo-rotations of many other manifolds. A Hamiltonian
pseudo-rotation is a Hamiltonian diffeomorphism with minimal possible
number of periodic points where minimality is interpreted in the
spirit of Arnold's conjecture. The actual definitions vary in general
(see \cite{CGG:CMD,GG:PR,Sh:HZ}), but for $M=\CP^n$ the requirement is
that $\varphi$ has exactly $n+1$ periodic points, which are then
necessarily the fixed points. Moreover, by \cite{GG:PR} and \cite{Sh:HZ}, all periodic points have one-dimensional local Floer homology, in particular, $b_\eps\big(\varphi^k\big)=n+1$ for any $\eps>0$. 

The simplest example of a Hamiltonian pseudo-rotation is a generic
element in a Hamiltonian circle or torus action with isolated fixed
points. However, in general, Hamiltonian pseudo-rotations can have
very interesting dynamics. For instance, Hamiltonian (aka,
area-preserving in this case) pseudo-rotations of $S^2=\CP^1$ with
exactly three invariant measures, which are then the two fixed points
and the area form, were constructed in \cite{AK}; see also
\cite{FK}. This construction was extended to symplectic toric
manifolds of any dimension in \cite{LRS}.

As an immediate consequence of Proposition
\ref{prop:AP}, we have the following.

\begin{Corollary}
  \label{cor:PR}
  Let $\varphi$ be a Hamiltonian pseudo-rotation of $\CP^n$ and let
  $L$ and $L'$ be Lagrangian submanifolds of $\CP^n$ as in Section
  \ref{sec:def}. Then for every $\eps>0$ the sequence
  $b_\eps\big(L,L^k\big)$ is bounded. In particular,
  $\hbar(\varphi; L,L')=0$.
\end{Corollary}

We do not know if we can replace barcode entropy by sequential barcode
entropy in Proposition \ref{prop:AP} and Corollary \ref{cor:PR}. Note
also that the absolute counterpart of Corollary \ref{cor:PR} is
obvious in contrast with Proposition \ref{prop:AP}:
$b_\eps\big(\varphi^k\big)=n+1$ for any $\eps>0$, and hence
$\hhbar(\varphi)=\hbar(\varphi)=0$, for a pseudo-rotation $\varphi$ of
$\CP^n$.

We conjecture that $\htop(\varphi)=0$ for any Hamiltonian
pseudo-rotation of $\CP^n$, and Corollary \ref{cor:PR} provides some
indirect evidence supporting this conjecture. In dimension two, the
conjecture follows immediately from, e.g., the results in \cite{Ka}
asserting that any area-preserving positive-entropy
$C^{1+\alpha}$-diffeomorphism of a compact surface must have a
horseshoe, and hence has infinitely many periodic points. Furthermore,
Hamiltonian pseudo-rotations $\varphi$ of $D^2$ or $\CP^n$ satisfying
a certain additional condition on the rotation number or the rotation
vector are known to be $C^0$-a.i.'s; see \cite{Br} and also \cite{A-Z}
for $D^2$ and \cite{GG:PR} for $\CP^n$ and \cite{JS} for Anosov-Katok
pseudo-rotations. Thus, in this case, $\htop(\varphi)=0$ by the
observation from \cite{A-Z} mentioned above.

\section{Proofs and refinements}
\label{sec:proofs}

In this section we prove Theorem \ref{thm:growth}, Proposition
\ref{prop:gamma} and also refine and prove Proposition
\ref{prop:AP}. Along the way we discuss some other ways to measure
barcode growth. The proof of Theorem \ref{thm:growth} hinges on a
construction from \cite{CGG:Entropy}, which we call a Lagrangian
tomograph and describe next.

\subsection{Lagrangian tomograph}
\label{sec:LT}
Let $L$ be a closed Lagrangian submanifold of a symplectic manifold
$M^{2n}$. A \emph{Lagrangian tomograph} is a map
$\Psi\colon B\times L\to M$, where $B=B^d$ is a ball of possibly very
large dimension $d$, which satisfies the following properties:

  \begin{itemize}

  \item[\reflb{LT1}{\rm{(i)}}] The map $\Psi$ is a submersion onto its
    image, the maps $\Psi_s:=\Psi|_{\{s\}\times L}$ are smooth
    embeddings for all $s\in B$ and $\Psi_0=\iota_L$ where
    $\iota_L \colon L \to M$ is the inclusion map;
    
  \item[\reflb{LT2}{\rm{(ii)}}] The images
    $L_s=\Psi\big(\{s\}\times L\big)$ are Lagrangian submanifolds of
    $M$ Hamiltonian isotopic to $L$.
\end{itemize}

Thus a Lagrangian tomograph is a family of Lagrangian submanifolds
$L_s$ which are parametrized by a ball $B$ and meet some additional
requirements. We call $d=\dim B$ the \emph{dimension of the
  tomograph}.  Note that we have $L_0=L$ by \ref{LT1}.  A Lagrangian
tomograph always exists for any closed Lagrangian submanifold $L$. In
fact, a Lagrangian tomograph of dimension $d$ exists if and only if
$L$ admits an immersion into $\R^d$; see \cite[Lemma
5.6]{CGG:Entropy}. We will need the following lemma.

\begin{Lemma}
\label{lemma:lip}
For some $C_{\hn}>0$ depending only on the tomograph, we have  
 \begin{equation}
      \label{eq:LT-Hofer}
    d_{\hn}(L_0,L_s)\leq C_{\hn}\|s\|. 
      \end{equation}
\end{Lemma}

\begin{proof}
  Since $B$ is compact it suffices to show that \eqref{eq:LT-Hofer}
  holds when $\| s\|$ is small. Fix a Weinstein neighborhood of $L=L_0$. Then, near $s=0$, each Lagrangian $L_s$ is
  given by the graph of some exact form $\alpha_s$. Let
  $\alpha_s = df_s$ be a smooth family of primitives; see
  \cite{ER}. (Note that, since $\Psi$ is smooth, the family $\alpha_s$
  is smooth in $s\in B$.) The claim \eqref{eq:LT-Hofer} follows from
  the following two inequalities
$$
d_{\hn}(L_0,L_s) \leq \max_L f_s - \min_L f_s \leq C_{\hn} \|s \|
$$ 
for some constant $C_{\hn}>0$. Here, in the first inequality one can
take $\pi^* f_s \colon T^*L \to \R$, where $\pi \colon T^*L \to L$ is the projection map, as the generating Hamiltonian. In fact,
the equality holds when $\| s\|$ is sufficiently small, \cite{Mi}, but
we do not need this fact. The second inequality follows from the
smoothness of the family $f_s$.
\end{proof}

Next, let $\tilde{L}$ be a closed $n$-dimensional submanifold of
$M$. Set
$$
N(s):=|L_s\cap \tilde{L}|\in [0,\,\infty].
$$
Since $\Psi$ is a submersion, $\Psi_s\pitchfork \tilde{L}$ for almost
all $s\in B$. Hence $N(s)<\infty$ almost everywhere and $N$ is an
integrable function on $B$.

Fix an auxiliary Riemannian metric on $M$ and let $ds$ be a smooth
measure on $B$, e.g., the standard Lebesgue measure. The key to the
proof of Theorem \ref{thm:growth} is the following observation.

\begin{Lemma}[Crofton's inequality; Lemma 5.3 in \cite{CGG:Entropy}]
  \label{lemma:Crofton}
  We have
  \[
  \int_B N(s)\,ds\leq C_{\Cr} \cdot \vol(\tilde{L}),
  \]
  where the constant $C_{\Cr}$ depends only on $ds$, $\Psi$ and the
  metric on $M$, but not on $\tilde{L}$.
\end{Lemma}

Of course, the lemma holds without the requirement that the
submanifolds $L_s$ are Lagrangian -- Condition \ref{LT1} is
sufficient. However, Condition \ref{LT2} is essential for the rest of
the proof and hence we included it in the definition of a Lagrangian
tomograph.

\subsection{Proof of Theorem \ref{thm:growth}}
\label{sec:pf-thm}
As are many arguments of this type, the proof is ultimately based on
Yomdin's theorem, \cite{Yo}, and quite similar to the proof of
\cite[Thm.\ 5.1]{CGG:Entropy}.  To prove the theorem, it suffices to
show that
\begin{equation}
  \labell{eq:htop-hbar-eps}
  \htop(\varphi)\geq 
  \hhbr_{\{\eps_k\}}(\varphi; L,L')
\end{equation}
for every subexponential sequence $\{\eps_k\}$ with
$\hhbr_{\{\eps_k\}}(\varphi; L,L') >0$. We will further assume that
$b_{\eps_k}\big(L_0,L^k\big) >0$ for all $\eps_k \in \{\eps_k\}$. This
can be always achieved without changing the growth rate via passing to
a subsequence.

Set $L^k:=\varphi^k(L')$. By Lemma \ref{lemma:Crofton}, we have a
sequence of integrable functions $N_k$ on $B$ such that,
$$
\int_B N_k(s)\,ds\leq C_{\Cr} \cdot \vol\big(L^k\big),
$$
where the constant $C_{\Cr}$ is independent of $k$.

For any sequence of balls $B_k\subset B$ of radius $\delta_k$ centered
at the origin, we have the following chain of inequalities
$$
C_{\Cr}\vol\big(L^k\big)\geq \int_B N_k\,ds \geq \int_{B_k}N_k\,ds
\geq \int_{B_k} b_{\eps_k/2}\big(L_s,L^k\big)\,ds,
$$
where in the last inequality we used \eqref{eq:intersections-b}. Let
$C_{\hn}$ be the constant from \eqref{eq:LT-Hofer}. By
\eqref{eq:intersections},
$$
b_{\eps_k/2}\big(L_s,L^k\big)\geq b_{\eta_k}\big(L_0,L^k\big)
$$
with $\eta_k=\eps_k/2+2C_{\hn}\delta_k$. Then, setting
$\delta_k=\eps_k/4C_{\hn}$, we obtain the inequality
$$
b_{\eps_k/2}\big(L_s,L^k\big)\geq b_{\eps_k}\big(L_0,L^k\big)
$$
as long as $s\in B_k$.  Therefore,
$$
\int_{B_k} b_{\eps_k/2}\big(L_s,L^k\big)\,ds \geq \int_{B_k}
b_{\eps_k}\big(L_0,L^k\big)\,ds=\vol\big(B_k\big)
b_{\eps_k}\big(L_0,L^k\big),
$$
where we took $ds$ to be the Lebesgue measure. To summarize,
\begin{equation}
  \label{eq:vol-b}
  \vol\big(L^k\big)\geq C_{\Cr}^{-1}\vol\big(B_k\big)
  b_{\eps_k}\big(L_0,L^k\big).
\end{equation}
Taking $\log^+$ of both sides and dividing by $k$, we have
\begin{equation}
  \label{eq:eps-k-ineq}
  \frac{\log^+\vol\big(L^k\big)}{k}\geq
  \frac{\log^+b_{\eps_k}\big(L_0,L^k\big)}{k}
  + d\cdot\frac{\log^+\eps_k}{k}+O(1/k),
\end{equation}
where $d=\dim B$. Due to the condition that $\{\eps_k\}$ is
subexponential, i.e., \eqref{eq:subexp2}, the second term on the right
goes to zero as $k\to\infty$. Thus, passing to the limit, we have
$$
\limsup_{k\to\infty}\frac{\log^+\vol\big(L^k\big)}{k} \geq
\hbar_{\{\eps_k\}}(\varphi; L,L').
$$
By Yomdin's theorem, \cite{Yo}, the left hand side is bounded from
above by $\htop(\varphi)$, and \eqref{eq:htop-hbar-eps} follows, which
concludes the proof of the theorem.  \qed

\begin{Remark}
  \label{rmk:threshold}
  This argument actually tells us a little bit more than Theorem
  \ref{thm:growth}. Focusing on the case of absolute entropy for the
  sake of simplicity, observe that \eqref{eq:vol-b} holds whenever
  $\eps_k$ is sufficiently small. The threshold for $\eps_k$ is
  determined by the tomograph. Then, by \eqref{eq:eps-k-ineq}, for any
  sequence $\{\eps_k\}$ which eventually becomes small enough, for
  instance whenever $\eps_k \to 0$, subexponential or not, we have
$$
\htop(\varphi)+d\cdot\limsup \frac{|\log^+\eps_k|}{k} \geq
\hbar_{\{\eps_k\}}(\varphi).
$$
\end{Remark}

\subsection{The shortest bar and total persistence}
\label{sec:applications}

In this section we briefly discuss some other ways to measure the size
of a barcode and relevant notions of barcode entropy. The first one is
centered around the shortest bar.

\begin{Proposition}
  \label{prop:htop-vs-p}
  Let $\varphi$ be a strongly non-degenerate Hamiltonian
  diffeomorphism $\varphi\colon M\to M$, where $M$ is closed and
  weakly monotone; cf.\ Section \ref{sec:def}. Denote by
  $\beta^{\min}_k$ the shortest bar for $\varphi^k$ and by $p(k)$
  the number of $k$-periodic points of $\varphi$. Then
  \begin{equation}
      \label{eq: htop-per2}
      \htop(\varphi)\geq
      \limsup \frac{\log^+p(k) - d\cdot |\log^+\beta^{\min}_k|}{k},
  \end{equation}
  where we can take as $d$ the minimal dimension of the Euclidean
  space which $M$ can be immersed into.
\end{Proposition}

To put this result in perspective, recall that when $\dim M>2$ there
is no clear-cut connection between the growth of $p(k)$ and
topological entropy in either direction; see \cite{Kal}.  Proposition
\ref{prop:htop-vs-p} resolves the problem to a certain extent by
providing a lower bound for topological entropy in terms of the
exponential growth rate of $p(k)$, but with a correction term coming
from the decay of $\beta^{\min}_k$. Furthermore, even in dimension
two, $p(k)$ can grow arbitrarily fast, even when $\varphi$ is strongly
non-degenerate; see \cite[Thm.\ 1.2]{As}. Then, by \eqref{eq:
  htop-per2}, $\beta^{\min}_k$ must go to zero super-exponentially and
\eqref{eq: htop-per2} still provides some information.
    
\begin{proof}
Observe that 
$$
p(k)=2b_{\eps}\big(\varphi^k\big)-\dim M
$$
for all $\eps< \beta^{\min}_k$. Then, as in Remark
\ref{rmk:threshold}, it is not hard to see from \eqref{eq:eps-k-ineq}
with $\eps_k=\min\{ c, \beta^{\min}_k\}$
that
\begin{equation}
  \label{eq:htop-per1}
  k\htop(\varphi)+d\cdot |\log^+ \min\{ c, 
  \beta^{\min}_k\}|
  \geq \log^+ p(k)+ o(k),
\end{equation}
where $c>0$ is the threshold mentioned in the remark.  On the other
hand, we have
$\beta^{\min}_k \leq \|\varphi^k\|_{\hn} \leq k \|\varphi\|_{\hn} $;
see \cite{Us2}.  It follows that
\begin{equation}
  \label{eq:htop-per1.5}
  |\log^+ \min\{ c,  \beta^{\min}_k\}|
  =  |\log^+  \beta^{\min}_k| + O(\log k). 
\end{equation}
Now, \eqref{eq: htop-per2} follows from \eqref{eq:htop-per1} and
\eqref{eq:htop-per1.5}.

\end{proof}

Another way to measure the size of a barcode is by looking at the
total persistence, i.e., the sum of the finite bar lengths taken to
some power $\alpha >0 $.  To be more specific, set
$$
\sigma_\alpha(\varphi)=\sum \beta_i(\varphi)^\alpha \in [0,\infty],
$$
where $\alpha > 0 $ is fixed and the sum is over all finite bars in
$\CB(\varphi)$. This is a Floer theoretic variant of the total
persistence; see, e.g., \cite{CSEHM,ST} and references therein.
Clearly, the above sum is finite when $\varphi$ is strongly
non-degenerate and this is the case we will focus on here. Then we
introduce a family of total persistence barcode entropies
$$
\hbar(\alpha, \varphi) := \limsup_{k \to \infty} \frac{\log^+
  \sigma_\alpha \big(\varphi^k\big)}{k}\in [0,\,\infty].
$$
In the strongly non-degenerate case, similarly to above, one can
regard $\hbar(\alpha, \varphi)$ as the exponential growth rate of
periodic points $p(k)$ counted with certain weights coming from the
barcode.

\begin{Proposition} 
  \label{prop:total_bar}
  Let $\varphi$ be a strongly non-degenerate Hamiltonian
  diffeomorphism of a closed and weakly monotone symplectic manifold
  $M$. Then, for $\hbar(\alpha, \varphi)$ defined as above:
  \begin{itemize}
  \item[\rm{(i)}] The function $\alpha\mapsto \hbar(\alpha, \varphi)$
    is (non-strictly) decreasing.

\item[\rm{(ii)}] We have
  $\hbar(\alpha, \varphi)\geq \hbar_\eps(\varphi)$ for all
  $\eps>0$. As a consequence,
  $\hbar(\alpha, \varphi)\geq \hbar(\varphi)$.

\item[\rm{(iii)}] For $\alpha\geq d$, where $d$ is the dimension of a
  tomograph,
  $$
  \hbar(\alpha, \varphi)\leq \htop(\varphi)<\infty.
  $$
\end{itemize}
\end{Proposition}

In the same vein, a variant of total persistence barcode entropy can
be defined for a Lagrangian or a pair of Lagrangians, and a similar
result holds in this setting.

\begin{Corollary}
  Assume that $M$ is a closed surface. Then, for $\alpha\geq 3$ and
  any strongly non-degenerate Hamiltonian diffeomorphism
  $\varphi\colon M\to M$, we have
  $$
  \hbar(\alpha,\varphi)=\hbar(\varphi)=\htop(\varphi).
  $$
\end{Corollary}

\begin{proof}
  Any closed surface can be immersed into $\R^3$ and even embedded
  when $M$ is orientable. Hence, when $M$ is such a surface, there
  exists a Lagrangian tomograph of dimension $d=3$; see
  \cite{CGG:Entropy}. Now, for $\alpha\geq 3$, in dimension two we
  have the chain of (in)equalities
  $$
  \hbar(\varphi)\leq \hbar(\alpha,\varphi)\leq
  \htop(\varphi)=\hbar(\varphi),
  $$
  where the first two inequalities follow from the proposition and in
  the last equality we use \cite[Thm.\ C]{CGG:Entropy}. Therefore,
  all three invariants are equal.
\end{proof}

\begin{Remark}
  We do not know if it can happen that
  $\hbar(\alpha, \varphi) =\infty$ for some $0 < \alpha<d$, e.g., for
  $\alpha=1$ which corresponds to the total bar length growth. But if
  it can, the infimum
$$
\overline{\alpha}_{\varphi}:= \inf
\{\alpha >0 \mid \hbar(\alpha, \varphi) <\infty\}
$$
would be a new Hausdorff dimension--type invariant of $\varphi$
associated with the barcodes $\CB\big(\varphi^k\big)$.
\end{Remark}

\begin{proof}[Proof of Proposition \ref{prop:total_bar}]
  We start by introducing a ``truncated'' version of the invariant
  $\hbar(\alpha, \varphi)$. For $b > 0$, set
$$
\sigma_{\alpha}^b(\varphi^k)=\sum \min\big\{b,
\beta_i\big(\varphi^k\big)\big\}^{\alpha}
$$
and let $\hbar^b(\alpha, \varphi)$ be the exponential growth rate (in
$k$) of $\sigma_{\alpha}^b(\varphi^k)$. More precisely,
$$
\hbar^b(\alpha, \varphi) := \limsup_{k \to \infty} \frac{\log^+
  \sigma_\alpha ^b(\varphi^k)}{k}\in [0,\,\infty].
$$
Clearly,
$$
\sigma_{\alpha}^b(\varphi)\leq\sigma_\alpha(\varphi),
$$
and hence
$$
\hbar^b(\alpha, \varphi)\leq \hbar(\alpha, \varphi).
$$
We claim that in fact
\begin{equation}
\label{eq:truncated}
\hbar^b(\alpha, \varphi) =\hbar (\alpha, \varphi) 
\end{equation}
for all $b>0$ and $\alpha >0$. Deferring the proof of
\eqref{eq:truncated} to the end, let us first prove the proposition.

For Part (i), we set $b=1$ and observe that
$\sigma_\alpha^1(\varphi^k)$ is decreasing in $\alpha >0 $. It follows
that the growth rate $\hbar^1(\alpha, \varphi)$, and hence
$\hbar(\alpha, \varphi)$ by \eqref{eq:truncated}, is also a decreasing
function of $\alpha>0$.

Furthermore, 
$$
\eps^\alpha b_\eps(\varphi)\leq \sigma_\alpha(\varphi)
$$
for all $\varphi$ and $\eps>0$. Applying this inequality to
$\varphi^k$, taking $\log^+$ and passing to the limit we obtain Part
(ii). (This argument is independent of \eqref{eq:truncated}.)

Next, let us focus on Part (iii). By Part (i), we may assume that
$\alpha=d$. Let $L^k$ be the graph of $\varphi^k$ and $L_0$ the
diagonal in $M\times M$. Thus for any $\eps>0$,
$b_\eps(\varphi^k)=b_\eps(L_0,L^k)$.

Fix a tomograph $L_s$ about $L_0$.
Then as in Section \ref{sec:pf-thm}
we have
$$
N_k(s) \geq b\big(L_s, L^k\big) \geq b_{C_H \|s\|}\big(L_s, L^k\big)
\geq b_{3C_H \|s\|}\big(L_0, L^k\big)
= b_{3C_H \|s\|}\big(\varphi^k\big)
$$
whenever $L_s \pitchfork L^k$; see
\eqref{eq:intersections} and Lemma \ref{lemma:lip}. Integrating both sides and using Lemma \ref{lemma:Crofton}, we obtain
\begin{equation}
\label{eq:total11}
C_{\Cr} \vol\big(L^k\big) \geq \int_B N_k(s)\, ds
\geq \int_B b_{3C_H \|s\|}\big(\varphi^k\big)\, ds. 
\end{equation}
A change of variables with $r=3C_H \| s \|$ yields 
\begin{equation}
\label{eq:total12}
\int_B b_{3C_H \|s\|}\big(\varphi^k\big)\, ds
= C \int_0^{r_0} b_r \big(\varphi^k\big) r^{d-1} \, dr
\end{equation}
for some $C>0$ and $r_0 >0$ independent of $k$. On the other hand, the
truncated sum $\sigma_d^{r_0}\big(\varphi^k\big)$ can also be written
as
\begin{equation}
\label{eq:total13}
\sigma_d^{r_0}\big(\varphi^k\big)
= d \cdot \sum \int_0^{r_0} \mathbbm{1}_{[0,
  \beta_i(\varphi^k))}(r)  r^{d-1} \, dr
= d\cdot \int_0^{r_0} b_r \big(\varphi^k\big) r^{d-1} \, dr,
\end{equation}
where $\mathbbm{1}_{[a,b)}(r)$ is the characteristic function of
$[a,b)$.  Here the second equality is a consequence of the identity
$$\sum \mathbbm{1}_{[0, \beta_i(\varphi^k))} (r) = b_{r}(\varphi^k).
$$
We combine \eqref{eq:total11}, \eqref{eq:total12} and
\eqref{eq:total13} to infer as in the proof of Theorem
\ref{thm:growth} that
$$
\htop(\varphi) \geq \hbar^{r_0}(d, \varphi) = \hbar (d, \varphi).
$$

It remains to prove the claim \eqref{eq:truncated}.  As in
\eqref{eq:htop-per1.5}, the claim relies on the linear upper bound
$\beta_i\big(\varphi^k\big) \leq k \|\varphi\|_{\hn} $; see
\cite{Us2}.  More precisely, we have
\begin{equation}
\label{eq:linear}
\sum_{\beta_i(\varphi^k) >b} b^\alpha
\leq \sum_{\beta_i\big(\varphi^k\big)  >b}
\beta_i\big(\varphi^k\big)^\alpha \leq
\sum_{\beta_i(\varphi^k) >b} (k \|\varphi\|_{\hn}) ^\alpha.
\end{equation}
Here all three terms and, in particular, the first two have the same
exponential growth rate. Now, \eqref{eq:truncated} is a consequence of
\eqref{eq:linear} and the general fact that
\begin{equation}
\label{eq:max}
\limsup \frac{\log^+ (p_k +q_k)}{k}
= \max \left\{ \limsup \frac{\log^+ p_k }{k} , \,
  \limsup \frac{\log^+ q_k }{k}  \right\},
\end{equation}
which holds for any real sequences $p_k\geq 0$ and $q_k \geq
0$. Namely, set
\[
p_k = \sum_{\beta_i(\varphi^k) >b} b^\alpha
\quad
\text{and}
\quad
p_k'= \sum_{\beta_i(\varphi^k) >b} \beta_i\big(\varphi^k\big)^\alpha,
\]
and also
\[
q_k= \sum_{\beta_i(\varphi^k) \leq b}\beta_i\big(\varphi^k\big)^\alpha.
\]
Then
\[
  \sigma_{\alpha}^b\big(\varphi^k\big)= p_k +q_k
\quad
\text{and}
\quad
  \sigma_{\alpha} \big(\varphi^k\big)= p_k' +q_k.
\]
By \eqref{eq:linear}, $p_k$ and $p_k'$ have the same exponential
growth rate.  Then it follows from \eqref{eq:max} that
$\sigma_{\alpha}^b\big(\varphi^k\big)$ and
$\sigma_{\alpha} \big(\varphi^k\big)$ have the same exponential growth
rate too.
\end{proof}

\begin{Remark}  
  Just as in Theorem \ref{thm:growth}, a similar construction can also
  be carried out in the relative setting for a pair of Lagrangians and
  an analogue of Proposition \ref{prop:total_bar} also holds in this
  case, with the same proof.
\end{Remark}  

\subsection{Proof of Proposition \ref{prop:gamma}}  
  \label{sec:pf-prop-gamma}
  
  For some $N\in \N$, $\varphi$ has more than $\dim \H_*(M)$
  hyperbolic $N$-periodic points. We denote the set of such points by
  $\mathcal{K}$. Thus $|\mathcal{K}|>\dim \H_*(M)$ and clearly
  $\mathcal{K}$ is a locally maximal hyperbolic set.  Furthermore,
  every point in $\mathcal{K}$ is also $\ell N$-periodic for all
  $\ell\in \N$. Then, arguing as in the proof of \cite[Thm.\
  B]{CGG:Entropy} and using \cite[Prop.\ 3.8 and 6.2]{CGG:Entropy}, we
  conclude that for a sufficiently small $\eps>0$ and any
  $\ell\in \N$,
  $$
  b_\eps\big(\varphi^{\ell N}\big)>\dim \H_*(M),
  $$
  and hence $\varphi^{\ell N}$ has a finite bar
  of length greater than $\eps>0$.

  Also recall that as is proved in \cite[Thm.\ A]{KS}, for any
  $\varphi$,
  \begin{equation}
  \label{eq:KS}
    \beta_{\max}(\varphi)\leq \gamma(\varphi),
  \end{equation}
  where the left-hand side is the \emph{boundary depth}, i.e., the
  longest finite bar in the barcode of $\varphi$. Thus, for a
  sufficiently small $\eps>0$,
  \begin{equation}
    \label{eq:eps-beta-gamma}
  \eps<\beta_{\max}\big(\varphi^{\ell N}\big)\leq \gamma
  \big(\varphi^{\ell N}\big).
\end{equation}

Next, arguing by contradiction, assume that there exists a sequence
$k_i\to\infty$ such that $\gamma\big(\varphi^{k_i}\big)\to 0$. We
claim that when $k_i<k_j$ are large enough, the difference $k_j-k_i$
is not divisible by $N$. This will imply the proposition; for then the
sequence $k_i$ would contain an infinite subsequence with the
difference between any two terms not divisible by $N$. This is
impossible because there are only finitely many residues modulo $N$.
  
  To prove the claim, assume the contrary:
  $$
  k_j-k_i=\ell N.
  $$
  Then by the triangle inequality for $\gamma$, we have
  $$
  \gamma\big(\varphi^{\ell N}\big) \leq \gamma\big(\varphi^{k_j}\big)
  + \gamma\big(\varphi^{-k_i}\big).
  $$
  Here, the right hand side becomes arbitrarily small when $k_i$ and
  $k_j$ are large, but the left hand side is bounded from below by
  $\eps>0$ by \eqref{eq:eps-beta-gamma}. This contradiction concludes
  the proof of the proposition. \qed
  
  \subsection{Proof and a refinement of Proposition \ref{prop:AP}}
  \label{sec:AP-pfs}
  Although we do not have any examples of $\gamma$-a.i.'s which are
  not pseudo-rotations, and hence $\gamma$-almost periodic when
  $M=\CP^n$, it is interesting to see that the zero-entropy part of
  the statement of the proposition holds under a less restrictive
  condition than $\gamma$-almost periodicity.

  To state the result, let us assume that $\varphi$ is a
  $\gamma$-a.i.\ and denote by $k_i=k_i(\eps)$ a strictly increasing
  sequence of integers such that
  $\gamma\big(\varphi^{k_i}\big)<\eps$. Furthermore, assume that there
  exists $\alpha<1$ independent of $\eps$ and such that
\begin{equation}
  \label{eq:gamma-gap}
    k_{i+1}-k_{i}\leq \alpha k_i,
  \end{equation}
  when $k_i$ is sufficiently large (depending on $\eps$). For
  instance, all $\gamma$-almost periodic Hamiltonian diffeomorphisms
  meet this requirement.

\begin{Proposition}
  \label{prop:AP2}
  Let $\varphi$ be a $\gamma$-a.i.\ satisfying
  \eqref{eq:gamma-gap}. Then $\hbar(\varphi; L,L')=0$ and
  $\hbar(\varphi)=0$.
\end{Proposition}

\begin{proof}[Proof of Propositions \ref{prop:AP} and \ref{prop:AP2}]
  The absolute case of the propositions concerning
  $b_\eps\big(\varphi^k\big)$ and $\hbar(\varphi)$ follows from the
  relative case by setting $L=L'$ to be the diagonal in $M\times M$
  and replacing $\varphi$ by $\id\times \varphi$. Hence we will focus
  on the relative case.

  The key to the proof is \cite[Thm.\ B]{KS} asserting, roughly
  speaking, that one can replace the Hofer norm in
  \eqref{eq:intersections} by the $\gamma$-norm. Namely, recall that
  for any two Lagrangian submanifolds $L'$ and $L''$ Hamiltonian
  isotopic to each other the $\gamma$-distance between $L'$ and $L''$
  is defined as
$$
\gamma(L',L''):=\inf\big\{\gamma(\psi)\mid \psi(L') = L''\big\}\leq
d_{\hn}(L',L'').
$$
Then, as a consequence of \cite[Thm.\ B]{KS}, we have the following
refinement of \eqref{eq:intersections}:
 \begin{equation}
   \label{eq:intersections2}
   b_{\eps+\delta}(L,L')\leq b_{\eps}(L,L'') \textrm{ when }
   \gamma(L',L'')<\delta/2.
\end{equation}

In the setting of Proposition \ref{prop:AP}, fix $\eps>0$. We claim
that there exist $N\in\N$ such that 
\begin{equation}
  \label{eq:bound}
  b_\eps\big(L,L^k\big)
  \leq \max_{0\leq \ell\leq N}b_{\eps/2}\big(L,L^\ell\big)
\end{equation}
for all sufficiently large $k \in \N$.  Indeed, since $\varphi$ is
$\gamma$-almost periodic, there exists a sequence of positive integers
\[
k_1<k_2<k_3<\ldots
\]
such that
$$
\gamma\big(\varphi^{k_i}\big)<\eps/4\textrm{ and } k_{i+1}-k_i\leq N
$$
for some $N$. Let $k \geq k_1$ and write $k=k_i+\ell$ with
$0\leq \ell \leq N$. Then
$$
L^k=\varphi^{k_i}\big(L^\ell\big) \textrm{ with }
\gamma\big(\varphi^{k_i}\big)<\eps/4.
$$
Hence, by \eqref{eq:intersections2},
$$
b_\eps\big(L,L^k\big)\leq b_{\eps/2}\big(L,L^\ell\big)
$$
for all $k \geq k_1$. This proves \eqref{eq:bound} and completes the
proof of Proposition \ref{prop:AP}.

Turning to the proof of Proposition \ref{prop:AP2}, for the sake of
brevity, set
$$
b_\eps(k):= b_\eps\big(L,L^k\big)\textrm{ and }
\hbar_\eps:=\hbar_\eps(\varphi; L,L').
$$
Let $k_i$ be a strictly increasing sequence of positive integers such
that
$$
\gamma\big(\varphi^{k_i}\big)<\eps/4 \textrm{ and } k_{i+1}-k_i\leq
\alpha k_i\textrm{ with } \alpha<1.
$$
As above, for $k \geq k_1$, write $k=k_i+\ell$ where now
$0\leq \ell \leq \alpha k_i$.  Then
$$
b_\eps(k)=b_\eps(k_i+\ell)\leq b_{\eps/2}(\ell).
$$
Furthermore, for any $\eta>\hbar_{\eps/2}$ and some constant $C$,
$$
\log^+ b_{\eps/2}(\ell)\leq\eta \ell+C\leq \eta\alpha k_i+ C\leq
\eta\alpha k+C.
$$
Combining these inequalities, we have
$$
\log^+ b_\eps(k)\leq \eta\alpha k +C.
$$
Dividing by $k$ and passing to the upper limit as $k\to\infty$, we see
that $\hbar_\eps\leq \eta\alpha$ for all $\eta>\hbar_{\eps/2}$, and
hence $\hbar_\eps\leq \alpha\hbar_{\eps/2}$. Equivalently,
$\hbar_{\eps/2}\geq \alpha^{-1}\hbar_\eps$ with
$0<\alpha<1$. Iterating this argument, we conclude that
$$
\hbar\geq \hbar_{2^{-i}\eps}\geq \alpha^{-i}\hbar_\eps\to \infty
\textrm{ as } i\to\infty
$$
unless $\hbar_\eps=0$. Since $\hbar<\infty$, we must have
$\hbar_\eps=0$, and hence $\hbar=0$.
\end{proof}

\begin{Remark}
  As we have pointed out in Section \ref{sec:AI+PR}, we do not know if
  we can replace barcode entropy by sequential barcode entropy in
  Propositions \ref{prop:AP} and~\ref{prop:AP2}.
\end{Remark}

\end{document}